\documentclass{article}
\usepackage{amsmath}
\usepackage{amsfonts}

\newtheorem{theorem}{Theorem}

\newtheorem{corollary}[theorem]{Corollary}

\newtheorem{definition}[theorem]{Definition}
\newtheorem{example}[theorem]{Example}

\newtheorem{lemma}[theorem]{Lemma}

\newtheorem{proposition}[theorem]{Proposition}
\newtheorem{remark}[theorem]{Remark}

\newenvironment{proof}[1][Proof]{\noindent \textbf{#1.} }{\  \rule{0.5em}{0.5em}}
\input{tcilatex}

\begin{document}

\title{Coordinate Finite Type Rotational Surfaces in Euclidean Spaces}
\author{B. K. Bayram, K. Arslan, N. \"{O}nen \  \& B. Bulca,}
\date{}
\maketitle

\begin{abstract}
Submanifolds of coordinate finite-type were introduced in \cite{HV1}. A
submanifold of a Euclidean space is called a coordinate finite-type
submanifold if its coordinate functions are eigenfunctions of $\Delta $. In
the present study we consider coordinate finite-type surfaces in $\mathbb{E}%
^{4}$. We give necessary and sufficient conditions for generalized rotation
surfaces in $\mathbb{E}^{4}$ to become coordinate finite-type. We also give
some special examples.
\end{abstract}

\section{Introduction}

\footnote{%
2000 AMS \textit{Mathematics Subject Classification}. 53B25, 53C40, 53C42
\par
\textit{Key words and phrases}: Surfaces of restricted type, Rotational
surface, Finite type surfaces}Let M be a connected $n-$dimensional
submanifold of a Euclidean space $\mathbb{E}^{m}$ equipped with the induced
metric. Denote $\Delta $ by the Laplacian of $M$ \ acting on smooth
functions on $M$ . This Laplacian can be extended in a natural way to $%
\mathbb{E}^{m}$ valued smooth functions on $M.$ Whenever the position vector 
$x$ of $M$ in $\mathbb{E}^{m}$ can be decomposed as a finite sum of $\mathbb{%
E}^{m}$-valued non-constant functions of $\Delta $, one can say that $M$ is
of \textit{finite type}. More precisely the position vector $x$ of $M$ can
be expresed in the form $x=x_{0}+\sum_{i=1}^{k}x_{i}$, where $x_{0}$ is a
constant map $x_{1},x_{2},...,x_{k}$ non-constant maps such that $\Delta
x=\lambda _{i}x_{i},$ $\lambda _{i}\in \mathbb{R}$, $1\leq i\leq k.$ If $%
\lambda _{1},\lambda _{2},...,\lambda _{k}$ are different, then $M$ is said
to be of $k$-type. Similarly, a smooth map $\phi $ of an $n$-dimensional
Riemannian manifold $M$ of $\mathbb{E}^{m}$ is said to be of finite type if $%
\phi $ is a finite sum of $\mathbb{E}^{m}$-valued eigenfunctions of $\Delta $
(\cite{Ch1}, \cite{Ch2}). For the position vector field $\overrightarrow{H}$
of M it is well known (see eg. \cite{Ch2}) that $\Delta x=-n\overrightarrow{H%
}$, which shows in particular that $M$ is a minimal submanifold in $\mathbb{E%
}^{m}$ if and only if its coordinate functions are harmonic. In \cite{Ta}
Takahasi proved that an n-dimensional submanifold of $\mathbb{E}^{m}$ is of
1-type (i.e., $\Delta x=$ $\lambda x$) if and only if it is either a minimal
submanifold of $\mathbb{E}^{m}$ or a minimal submanifold of some hypersphere
of $\mathbb{E}^{m}.$ As a generalization of T. Takahashi's condition, O.
Garay considered in \cite{Go1}, submanifolds of Euclidean space whose
position vector field $x$ satisfies the differential equation $\Delta x=Ax$,
for some $m\times m$ diagonal matrix $A$. Garay called such submanifolds 
\textit{coordinate finite type submanifolds}. Actually coordinate finite
type submanifolds are finite type submanifolds whose type number s are at
most $m$. Each coordinate function of a coordinate finite type submanifold $%
m $ is of 1-type, since it is an eigenfunction of the Laplacian \cite{HV1}.

In \cite{GM} \ by G. Ganchev and V. Milousheva considered the surface M
generated by a W-curve $\gamma $ in $\mathbb{E}^{4}.$ They have shown that
these generated surfaces are a special type of rotation surfaces which are
introduced first by C. Moore in $1919$ (see \cite{Mo}). Vranceanu surfaces
in $\mathbb{E}^{4}$ are the special type of these surfaces \cite{Vr}.

This paper is organized as follows: Section $2$ gives some basic concepts of
the surfaces in $\mathbb{E}^{4}$. Section $3$ tells about the generalized
surfaces in $\mathbb{E}^{4}$. Further this section provides some basic
properties of surfaces in $\mathbb{E}^{4}$ and the structure of their
curvatures. In the final section we consider coordinate finite type surfaces
in euclidean spaces. We give necessary and sufficient conditions for
generalized rotation surfaces in $\mathbb{E}^{4}$ to become coordinate
finite type.

\section{Basic Concepts}

Let $M$ be a smooth surface in $\mathbb{E}^{n}$ given with the patch $X(u,v)$
: $(u,v)\in D\subset \mathbb{E}^{2}$. The tangent space to $M$ at an
arbitrary point $p=X(u,v)$ of $M$ span $\left \{ X_{u},X_{v}\right \} $. In
the chart $(u,v)$ the coefficients of the first fundamental form of $M$ are
given by 
\begin{equation}
E=<X_{u},X_{u}>,F=\left \langle X_{u},X_{v}\right \rangle ,G=\left \langle
X_{v},X_{v}\right \rangle ,  \label{A2*}
\end{equation}%
where $\left \langle ,\right \rangle $ is the Euclidean inner product. We
assume that $W^{2}=EG-F^{2}\neq 0,$ i.e. the surface patch $X(u,v)$ is
regular.\ For each $p\in M$, consider the decomposition $T_{p}\mathbb{E}%
^{n}=T_{p}M\oplus T_{p}^{\perp }M$ where $T_{p}^{\perp }M$ is the orthogonal
component of $T_{p}M$ in $\mathbb{E}^{n}$. Let $\overset{\sim }{\nabla }$ be
the Riemannian connection of $\mathbb{E}^{4}$. Given unit local vector
fields $X_{1},$ $X_{2}$ tangent to $M$.

Let $\chi (M)$ and $\chi ^{\perp }(M)$ be the space of the smooth vector
fields tangent to $M$ and the space of the smooth vector fields normal to $M$%
, respectively. Consider the second fundamental map: $h:\chi (M)\times \chi
(M)\rightarrow \chi ^{\perp }(M);$%
\begin{equation}
h(X_{i},X_{_{j}})=\widetilde{\nabla }_{X_{_{i}}}X_{_{j}}-\nabla
_{X_{_{i}}}X_{_{j}}\text{ \  \  \ }1\leq i,j\leq 2.  \label{A3}
\end{equation}%
where $\widetilde{\nabla }$ is the induced. This map is well-defined,
symmetric and bilinear.

For any arbitrary orthonormal normal frame field $\left \{
N_{1},N_{2},...,N_{n-2}\right \} $ of $M$, recall the shape operator $A:\chi
^{\perp }(M)\times \chi (M)\rightarrow \chi (M);$%
\begin{equation}
A_{N_{i}}X_{j}=-(\widetilde{\nabla }_{X_{j}}N_{k})^{T},\text{ \  \  \ }%
X_{j}\in \chi (M),\text{ }1\leq k\leq n-2  \label{A4}
\end{equation}%
This operator is bilinear, self-adjoint and satisfies the following equation:%
\begin{equation}
\left \langle A_{N_{k}}X_{j},X_{i}\right \rangle =\left \langle
h(X_{i},X_{j}),N_{k}\right \rangle =h_{ij}^{k}\text{, }1\leq i,j\leq 2.
\label{A5}
\end{equation}

The equation (\ref{A3}) is called Gaussian formula, and%
\begin{equation}
h(X_{i},X_{j})=\overset{n-2}{\underset{k=1}{\sum }}h_{ij}^{k}N_{k},\  \  \  \  \
1\leq i,j\leq 2  \label{A6}
\end{equation}%
where $c_{ij}^{k}$ are the coefficients of the second fundamental form.

Further, the Gaussian and mean curvature vector of a regular patch $X(u,v)$
are given by%
\begin{equation}
K=\sum_{k=1}^{n-2}(h_{11}^{k}h_{22}^{k}-(h_{12}^{k})^{2}),  \label{A7}
\end{equation}%
and 
\begin{equation}
H=\frac{1}{2}\sum_{k=1}^{n-2}(h_{11}^{k}+h_{22}^{k})N_{k},  \label{A8}
\end{equation}%
respectively, where $h$ is the second fundamental form of $M$. Recall that a
surface $M$ is said to be \textit{minimal} if its mean curvature vector
vanishes identically \cite{Ch1}. For any real function $f$ on $M$ the
Laplacian of $f$ is defined by

\begin{equation}
\Delta f=-\sum_{i}(\widetilde{\nabla }_{e_{i}}\widetilde{\nabla }_{e_{i}}f-%
\widetilde{\nabla }_{\nabla _{e_{i}}e_{i}}f).  \label{A9}
\end{equation}%
\ 

\section{Generalized Rotation Surfaces in $\mathbb{E}^{4}$}

Let $\gamma =\gamma (s):I\rightarrow \mathbb{E}^{4}$ be a W-curve in
Euclidean $4$-space $\mathbb{E}^{4}$ parametrized as follows:%
\begin{equation*}
\gamma (v)=(a\cos cv,a\sin cv,b\cos dv,b\sin dv),\text{ }0\leq v\leq 2\pi ,
\end{equation*}%
where $a,b,c,d$ are constants $(c>0,d>0).$ In \cite{GM} G. Ganchev and V.
Milousheva considered the surface $M$ generated by the curve $\gamma $ with
the following surface patch: 
\begin{equation}
X(u,v)=(f(u)\cos cv,f(u)\sin cv,g(u)\cos dv,g(u)\sin dv),  \label{C2}
\end{equation}%
where $u\in J,0\leq v\leq 2\pi ,$ $f(u)$ and $g(u)$ are arbitrary smooth
functions satisfying 
\begin{equation*}
c^{2}f\text{ }^{2}+d^{2}g^{2}>0\text{ and }(f^{\text{ }\prime \text{ }%
})^{2}+(g^{\text{ }\prime \text{ }})^{2}>0.
\end{equation*}

These surfaces are first introduced by C. Moore in \cite{Mo} , called 
\textit{general rotation surfaces.}

We choose an orthonormal frame $\left \{ e_{1},e_{2},e_{3},e_{4}\right \} $
such that $e_{1},e_{2}$ are tangent to $M$ and $e_{3},e_{4}$ normal to $M$
in the following (see, \cite{GM}): \qquad 
\begin{eqnarray}
e_{1} &=&\frac{X_{u}}{\left \Vert X_{u}\right \Vert },\text{ }e_{2}=\frac{%
X_{v}}{\left \Vert X_{u}\right \Vert }\   \notag \\
e_{3} &=&\frac{1}{\sqrt{(f^{\prime })^{2}+(g^{\prime })^{2}}}(g^{\prime
}\cos cv,g^{\prime }\sin cv,-f^{\prime }\cos dv,-f^{\prime }\sin dv),
\label{C2*} \\
e_{4} &=&\frac{1}{\sqrt{c^{2}f^{2}+d^{2}g^{2}}}(-dg\sin cv,dg\cos cv,cf\sin
dv,-cf\cos dv).  \notag
\end{eqnarray}

Hence the coefficients of the first fundamental form of the surface are%
\begin{eqnarray}
E &=&\left \langle X_{u},X_{u}\right \rangle =(f^{\prime })^{2}+(g^{\prime
})^{2}  \notag \\
F &=&\left \langle X_{u},X_{v}\right \rangle =0  \label{C3} \\
G &=&\left \langle X_{v},X_{v}\right \rangle =c^{2}f^{2}+d^{2}g^{2}  \notag
\end{eqnarray}%
where $\left \langle ,\right \rangle $ is the standard scalar product in $%
\mathbb{E}^{4}.$ Since 
\begin{equation*}
EG-F^{2}=\left( (f^{\prime })^{2}+(g^{\prime })^{2}\right) \left(
c^{2}f^{2}+d^{2}g^{2}\right)
\end{equation*}%
does not vanish, the surface patch $X(u,v)$ is regular. Then with respect to
the frame field $\left \{ e_{1},e_{2},e_{3},e_{4}\right \} ,$ the Gaussian
and Weingarten formulas (\ref{A3})-(\ref{A4}) of $M$ look like (see, \cite%
{DT}); 
\begin{eqnarray}
\tilde{\nabla}_{e_{1}}e_{1} &=&-A(u)e_{2}+h_{11}^{1}e_{3},  \notag \\
\tilde{\nabla}_{e_{1}}e_{2} &=&A(u)e_{1}+h_{12}^{2}e_{4},  \label{C4} \\
\tilde{\nabla}_{e_{2}}e_{2} &=&h_{22}^{1}e_{3},  \notag \\
\tilde{\nabla}_{e_{2}}e_{1} &=&h_{12}^{2}e_{4},  \notag
\end{eqnarray}%
and%
\begin{eqnarray}
\tilde{\nabla}_{e_{1}}e_{3} &=&-h_{11}^{1}e_{1}+B(u)e_{4},  \notag \\
\tilde{\nabla}_{e_{1}}e_{4} &=&-h_{12}^{2}e_{2}-B(u)e_{3},  \label{C5} \\
\tilde{\nabla}_{e_{2}}e_{3} &=&-h_{22}^{1}e_{2},  \notag \\
\tilde{\nabla}_{e_{2}}e_{4} &=&-h_{12}^{2}e_{1},  \notag
\end{eqnarray}%
where\qquad 
\begin{eqnarray}
A(u) &=&\frac{c^{2}ff^{\prime }+d^{2}gg^{\prime }}{\sqrt{(f^{\prime
})^{2}+(g^{\prime })^{2}}(c^{2}f^{2}+d^{2}g^{2})},  \notag \\
B(u) &=&\frac{cd(ff^{\prime }+gg^{\prime })}{\sqrt{(f^{\prime
})^{2}+(g^{\prime })^{2}}(c^{2}f^{2}+d^{2}g^{2})},  \notag \\
h_{11}^{1} &=&\frac{d^{2}f^{\prime }g-c^{2}fg^{\prime }}{\sqrt{(f^{\prime
})^{2}+(g^{\prime })^{2}(}c^{2}f^{2}+d^{2}g^{2})},  \notag \\
h_{22}^{1} &=&\frac{g^{\prime }f^{\prime \prime }-f^{\prime }g^{\prime
\prime }}{\left( (f^{\prime })^{2}+(g^{\prime })^{2}\right) ^{\frac{3}{2}}},
\label{C6} \\
h_{12}^{2} &=&\frac{cd(f^{\prime }g-fg^{\prime })}{\sqrt{(f^{\prime
})^{2}+(g^{\prime })^{2}(}c^{2}f^{2}+d^{2}g^{2})},  \notag \\
h_{11}^{2} &=&h_{22}^{2}=h_{12}^{1}=0.  \notag
\end{eqnarray}%
are the differentiable functions. Using (\ref{A7})-(\ref{A8}) with (\ref{C6}%
) one can get the following results;

\begin{proposition}
\cite{ABBO} Let $M$ be a generalized rotation surface given by the
parametrization (\ref{C2}), then the Gaussian curvature of $M$ is%
\begin{equation*}
K=\frac{(c^{2}f{}^{\text{ }2}+d^{2}g{}^{2})(g^{\text{ }\prime }f{}^{\text{ }%
\prime \prime }\text{-}f^{\text{ }\prime }g^{\text{ }\prime \prime
})(d^{2}g^{\text{ }}f{}^{\text{ }\prime }\text{-}c^{2}fg^{\text{ }\prime })%
\text{-}c^{2}d^{2}(g^{\text{ }}f{}^{\text{ }\prime }\text{-}fg^{\text{ }%
\prime })^{2}((f{}^{\text{ }\prime })^{2}+(g{}^{\prime })^{2})}{((f{}^{\text{
}\prime })^{2}+(g{}^{\prime })^{2})^{2}(c^{2}f{}^{\text{ }%
2}+d^{2}g{}^{2})^{2}}.
\end{equation*}
\end{proposition}

An easy consequence of Proposition $1$ is the following.

\begin{corollary}
\cite{ABBO} The generalized rotation surface given by the parametrization (%
\ref{C2}) has vanishing Gaussian curvature if and only if the following
equation%
\begin{equation*}
(c^{2}f{}^{\text{ }2}+d^{2}g{}^{2})(g^{\text{ }\prime }f{}^{\text{ }\prime
\prime }\text{-}f^{\text{ }\prime }g^{\text{ }\prime \prime })(d^{2}g^{\text{
}}f{}^{\text{ }\prime }\text{-}c^{2}fg^{\text{ }\prime })\text{-}%
c^{2}d^{2}(g^{\text{ }}f{}^{\text{ }\prime }\text{-}fg^{\text{ }\prime
})^{2}((f{}^{\text{ }\prime })^{2}+(g{}^{\prime })^{2})=0,
\end{equation*}%
holds.
\end{corollary}

The following results are well-known;

\begin{proposition}
\cite{ABBO} Let $M$ be a generalized rotation surface given by the
parametrization (\ref{C2}), then the mean curvature vector of $M$ is%
\begin{eqnarray*}
\overrightarrow{H} &=&\frac{1}{2}(h_{11}^{1}+h_{22}^{1})e_{3} \\
&=&\left( \frac{(c^{2}f{}^{\text{ }2}+d^{2}g{}^{2})(g^{\text{ }\prime }f{}^{%
\text{ }\prime \prime }-f^{\text{ }\prime }g^{\text{ }\prime \prime
})+(d^{2}g^{\text{ }}f{}^{\text{ }\prime }-c^{2}fg^{\text{ }\prime })((f{}^{%
\text{ }\prime })^{2}+(g{}^{\prime })^{2})}{2((f{}^{\text{ }\prime
})^{2}+(g{}^{\prime })^{2})^{3/2}(c^{2}f{}^{\text{ }2}+d^{2}g{}^{2})}\right)
e_{3}.
\end{eqnarray*}
\end{proposition}

An easy consequence of Proposition $3$ is the following.

\begin{corollary}
\cite{ABBO} The generalized rotation surface given by the parametrization (%
\ref{C2}) is minimal surface in $\mathbb{E}^{4}$ if and only if the equation 
$\ $%
\begin{equation*}
(c^{2}f{}^{\text{ }2}+d^{2}g{}^{2})(g^{\text{ }\prime }f{}^{\text{ }\prime
\prime }-f^{\text{ }\prime }g^{\text{ }\prime \prime })+(d^{2}g^{\text{ }%
}f{}^{\text{ }\prime }-c^{2}fg^{\text{ }\prime })((f{}^{\text{ }\prime
})^{2}+(g{}^{\prime })^{2})=0,
\end{equation*}%
holds.
\end{corollary}

\begin{definition}
The generalized rotation surface given by the parametrization%
\begin{equation}
f(u)=r(u)\cos u,\text{ }g\text{ }(u)=r(u)\sin u,\text{ }c=1,d=1.  \label{C18}
\end{equation}%
is called \textit{Vranceanu rotation surface} in Euclidean 4-space $\mathbb{E%
}^{4}$ \cite{Vr}$.$
\end{definition}

\begin{remark}
Substituting (\ref{C18}) into the equation given in Corollary $2$ we obtain
the condition for Vranceanu rotation surface which has vanishing Gaussian
curvature;%
\begin{equation}
r(u)r^{\prime \prime }(u)-(r^{\prime }(u))^{2}=0.  \label{C19}
\end{equation}%
Further, and easy calculation shows that $r(u)=\lambda e^{\mu u},(\lambda
,\mu \in R)$ \ is the solution is this second degree equation. So, we get
the following result.
\end{remark}

\begin{corollary}
\cite{Yo} Let $M$ is a Vranceanu rotation surface in Euclidean 4-space. If $%
M $ has vanishing Gaussian curvature, then $r(u)=\lambda e^{\mu u}$, where $%
\lambda $ and $\mu $ are real constants. For the case, $\lambda =1,$ $\mu
=0,r(u)=1,$ the surface $M$ is a Clifford torus, that is it is the product
of two plane circles with same radius.
\end{corollary}

\begin{corollary}
\cite{ABBO} Let $M$ is a Vranceanu rotation surface in Euclidean 4-space. If 
$M$ is minimal then%
\begin{equation*}
r(u)r^{\prime \prime }(u)-3(r^{\prime }(u))^{2}-2r(u)^{2}=0.
\end{equation*}%
holds.
\end{corollary}

\begin{corollary}
\cite{ABBO} Let $M$ is a Vranceanu rotation surface in Euclidean 4-space. If 
$M$ is minimal then 
\begin{equation}
r(u)=\frac{\pm 1}{\sqrt{a\sin 2u-b\cos 2u}},  \label{C20}
\end{equation}%
where, $a$ and $b$ are real constants.
\end{corollary}

\begin{definition}
The surface patch $X(u,v)$ is called pseudo-umbilical if the shape operator
with respect to $H$ is proportional to the identity (see, \cite{Ch1}). An
equivalent condition is the following:%
\begin{equation}
<h(X_{i},X_{j}),H>=\lambda ^{2}<X_{i},X_{j}>,  \label{C25}
\end{equation}%
where, $\lambda =\left \Vert H\right \Vert .$ It is easy to see that each
minimal surface is pseudo-umbilical.
\end{definition}

The following results are well-known;

\begin{theorem}
\cite{ABBO} Let $M$ be a generalized rotation surface given by the
parametrization (\ref{C2}) is pseudo-umbilical then 
\begin{equation}
(c^{2}f{}^{\text{ }2}+d^{2}g{}^{2})(g^{\text{ }\prime }f{}^{\text{ }\prime
\prime }-f^{\text{ }\prime }g^{\text{ }\prime \prime })-(d^{2}g^{\text{ }%
}f{}^{\text{ }\prime }-c^{2}fg^{\text{ }\prime })((f{}^{\text{ }\prime
})^{2}+(g{}^{\prime })^{2})=0.  \label{C26}
\end{equation}
\end{theorem}

The converse statement of \ Theorem 11 is also valid.

\begin{corollary}
\cite{ABBO} Let $M$ be a Vranceanu rotation surface in Euclidean 4-space. If 
$M$ pseudo-umbilical then $r(u)=\lambda e^{\mu u}$, where $\lambda $ and $%
\mu $ are real constants.
\end{corollary}

\subsection{Coordinate Finite Type Surfaces in Euclidean Spaces}

In the present section we consider coordinate finite type surfaces in
Euclidean spaces $\mathbb{E}^{n+2}$. A surface $M$ in Euclidean $m$-space is
called coordinate finite type if the position vector field $X$ satisfies the
differential equation 
\begin{equation}
\Delta X=AX,  \label{D1*}
\end{equation}%
for some $m\times m$ diagonal matrix $A$. Using the Beltrami formula's $%
\Delta X=-2\overrightarrow{H}$, with (\ref{A8}) one can get 
\begin{equation}
\Delta X=-\sum_{k=1}^{n}(h_{11}^{k}+h_{22}^{k})N_{k}.  \label{D2}
\end{equation}

So, using \ (\ref{D1*}) with (\ref{D2}) the coordinate finite type condition
reduces to

\begin{equation}
AX=-\sum_{k=1}^{n}(h_{11}^{k}+h_{22}^{k})N_{k}  \label{D2*}
\end{equation}

For a non-compact surface in $\mathbb{E}^{4}$ O.J.Garay obtained the
following:

\begin{theorem}
\cite{Go2} The only coordinate finite type surfaces in Euclidean 4-space $%
\mathbb{E}^{4}$ with constant mean curvature are the open parts of the
following surfaces:

$i)$ a minimal surface in $\mathbb{E}^{4},$

$ii)$ a minimal surface in some hypersphere $S^{3}(r),$

$iii)$ a helical cylinder,

$iv)$ a flat torus $S^{1}(a)\times S^{1}(b)$ in some hypersphere $S^{3}(r)$.
\end{theorem}

In \cite{CDVV} Chen-Dillen-Verstraelen-Vrancken proved the following theorem;

\begin{theorem}
\cite{CDVV} Assume $M$ is a surface in $\mathbb{E}^{4}$ that is immersed in $%
S^{3}(r)$ and has constant mean curvature. Then $M$ is of restricted type if
and only if $M$ is one of the following:

$i)$ an open part of a minimal surface of $S^{3}(r),$

$ii)$ an open part of $S^{2}(r^{\prime })$ for $0<r^{\prime }\leq r$ ,

$iii)$ an open part of the product of two circles $S^{1}(a)\times S^{1}(b),$
where $a,b>0$ and $a^{2}+b^{2}=r^{2}.$
\end{theorem}

\subsection{ Surface of Revolution of Coordinate Finite Type}

A surface in $\mathbb{E}^{3}$ is called a surface of revolution if it is
generated by a curve $C$ on a plane $\Pi $ when $\Pi $ is rotated around a
straight line $L$ in $\Pi .$ By choosing $\Pi $ to be the $xz$-plane and
line $L$ to be the $x$ axis the surface of revolution can be parameterized by%
\begin{equation}
X(u,v)=\left( f(u),g(u)\cos v,g(u)\sin v\right) ,  \label{F1}
\end{equation}%
where $f(u)$ and $g(u)$ are arbitrary smooth functions. We choose an
orthonormal frame $\left \{ e_{1},e_{2},e_{3}\right \} $ such that $%
e_{1},e_{2} $ are tangent to $M$ and $e_{3}$ normal to $M$ in the following:
\qquad 
\begin{equation}
e_{1}=\frac{X_{u}}{\left \Vert X_{u}\right \Vert },\text{ }e_{2}=\frac{X_{v}}{%
\left \Vert X_{v}\right \Vert },\ e_{3}=\frac{1}{\sqrt{(f^{\prime
})^{2}+(g^{\prime })^{2}}}(g^{\prime },-f^{\prime }\cos v,-f^{\prime }\sin
v),  \label{f1}
\end{equation}

By covariant differentiation with respect to $e_{1},e_{2}$ a straightforward
calculation gives 
\begin{eqnarray}
\tilde{\nabla}_{e_{1}}e_{1} &=&h_{11}^{1}e_{3},  \notag \\
\tilde{\nabla}_{e_{2}}e_{2} &=&-A(u)e_{1}+h_{22}^{2}e_{3},  \label{F2} \\
\tilde{\nabla}_{e_{2}}e_{1} &=&A(u)e_{2},  \notag \\
\tilde{\nabla}_{e_{1}}e_{2} &=&0,  \notag
\end{eqnarray}%
where\qquad 
\begin{eqnarray}
A(u) &=&\frac{g^{\prime }}{g\sqrt{(f^{\prime })^{2}+(g^{\prime })^{2}}}, 
\notag \\
h_{11}^{1} &=&\frac{g^{\prime }f^{\prime \prime }-f^{\prime }g^{\prime
\prime }}{\left( (f^{\prime })^{2}+(g^{\prime })^{2}\right) ^{\frac{3}{2}}},
\label{F3} \\
h_{22}^{1} &=&\frac{f^{\prime }}{g\sqrt{(f^{\prime })^{2}+(g^{\prime })^{2}}}%
,  \notag \\
h_{12}^{1} &=&0.  \notag
\end{eqnarray}%
are the differentiable functions. Using (\ref{A7})-(\ref{A8}) with (\ref{F3}%
) one can get%
\begin{equation}
\overrightarrow{H}=\frac{1}{2}\left( h_{11}^{1}+h_{22}^{1}\right) e_{3}
\label{F4}
\end{equation}%
where $h_{11}^{1}$ and $h_{22}^{1}$ are the coefficients of the second
fundamental form given in (\ref{F3}).

A surface of revolution defined by (\ref{F1}) is said to be of polynomial
kind if $f(u)$ and $g(u)$ are polynomial functions in $u$ and it is said to
be of rational kind if $f$ is a rational function in $g,$ i.e., $f$ \ is the
quotient of two polynomial functions in $g$ \cite{Ch3}.

For finite type surfaces of revolution B.Y. Chen and S. Ishikawa obtained in 
\cite{CI} the following results;

\begin{theorem}
\cite{CI} Let $M$ be a surface of revolution of polynomial kind. Then $M$ is
a surface of finite type if and only if either it is an open portion of a
plane or it is an open portion of a circular cylinder.
\end{theorem}

\begin{theorem}
\cite{CI} Let $M$ be a surface of revolution of rational kind. Then $M$ is a
surface of finite type if and only if $M$ is an open portion of a plane.
\end{theorem}

T. Hasanis and T. Vlachos proved the following.

\begin{theorem}
\cite{HV1} Let $M$ be a surface of revolution. If $M$ has constant mean
curvature and is of finite type then $M$ is an open portion of a plane, of a
sphere or of a circular cylinder.
\end{theorem}

We proved the following result;

\begin{lemma}
Let $M$ be a surface of revolution given with the parametrization (\ref{F1}%
). Then $M$ is a surface of \textit{coordinate finite type if and only if }%
diagonal matrix $A$ is of the form%
\begin{equation}
A=\left[ 
\begin{array}{ccc}
a_{11} & 0 & 0 \\ 
0 & a_{22} & 0 \\ 
0 & 0 & a_{33}%
\end{array}%
\right]  \label{F5}
\end{equation}%
where%
\begin{eqnarray}
a_{11} &=&\frac{-g^{\prime }(g\left( g^{\prime }f^{\prime \prime }-f^{\prime
}g^{\prime \prime })+f^{\prime }\left( (f^{\prime })^{2}+(g^{\prime
})^{2}\right) \right) }{fg\left( (f^{\prime })^{2}+(g^{\prime })^{2}\right)
^{2}}  \label{F6} \\
a_{22} &=&a_{33}=\frac{f^{\prime }\left( g(g^{\prime }f^{\prime \prime
}-f^{\prime }g^{\prime \prime })+f^{\prime }\left( (f^{\prime
})^{2}+(g^{\prime })^{2}\right) \right) }{g^{2}\left( (f^{\prime
})^{2}+(g^{\prime })^{2}\right) ^{2}}  \notag
\end{eqnarray}%
are differentiable functions.
\end{lemma}

\begin{proof}
Assume that the surface of revolution $M$ given with the parametrization (%
\ref{F1}) is of \textit{coordinate finite type}$.$ Then, from the equality\ (%
\ref{D2*}) 
\begin{equation}
\Delta X=-(h_{11}^{1}+h_{22}^{1})e_{3}.  \label{F7}
\end{equation}%
Further, substituting (\ref{F3}) into (\ref{F7}) and using (\ref{f1}) we get
the result.
\end{proof}

\begin{remark}
If the diagonal matrix $A$ is \ equivalent to a zero matrix then $M$ becomes
minimal. So the surface of revolution $M$ is either an open portion of a
plane or an open portion of a catenoid.
\end{remark}

Minimal rotational surfaces are of coordinate finite type.

For the non-minimal case we obtain the following result;

\begin{proposition}
Let $M$ be a non-minimal surface of revolution given with the
parametrization (\ref{F1}). If $M$ is \textit{coordinate finite type surface
then} 
\begin{equation}
ff^{\text{ }\prime }+\lambda gg\prime =0  \label{F8}
\end{equation}%
holds, where $\lambda $ is a nonzero constant.
\end{proposition}

\begin{proof}
Suppose that the entries of the diagonal matrix $A$ are real constants. Then
using (\ref{F6}) one can get the following differential equations 
\begin{eqnarray*}
\frac{-g^{\prime }\left( g(g^{\prime }f^{\prime \prime }-f^{\prime
}g^{\prime \prime })+f^{\prime }\left( (f^{\prime })^{2}+(g^{\prime
})^{2}\right) \right) }{fg\left( (f^{\prime })^{2}+(g^{\prime })^{2}\right)
^{2}} &=&c_{1} \\
\frac{f^{\prime }\left( g(g^{\prime }f^{\prime \prime }-f^{\prime }g^{\prime
\prime })+f^{\prime }\left( (f^{\prime })^{2}+(g^{\prime })^{2}\right)
\right) }{g^{2}\left( (f^{\prime })^{2}+(g^{\prime })^{2}\right) ^{2}}
&=&c_{2}.
\end{eqnarray*}%
where $c_{1},c_{2}$ are nonzero real constants. Further, substituting one
into another we obtain the result.
\end{proof}

\begin{example}
The round sphere given with the parametrization $f(u)=r\cos u,$ $g(u)=r\sin
u $ satisfies the equality (\ref{F8}). So it is a coordinate finite type
surface.
\end{example}

\begin{example}
The cone $f(u)=g(u)$ satisfies the equality (\ref{F8}). So it is a
coordinate finite type surface.
\end{example}

\subsection{Generalized Rotation Surfaces of Coordinate Finite Type}

In the present section we consider generalized rotation surfaces of
coordinate finite type surfaces in Euclidean 4-spaces $\mathbb{E}^{4}$.

We proved the following result;

\begin{lemma}
Let $M$ be a generalized rotation surface given with the parametrization (%
\ref{C2}). Then $M$ is a surface of \textit{coordinate finite type if and
only if }diagonal matrix $A$ is of the form%
\begin{equation}
A=\left[ 
\begin{array}{cccc}
a_{11} & 0 & 0 & 0 \\ 
0 & a_{22} & 0 & 0 \\ 
0 & 0 & a_{33} & 0 \\ 
0 & 0 & 0 & a_{44}%
\end{array}%
\right]  \label{D3}
\end{equation}%
\bigskip where%
\begin{equation}
\begin{array}{l}
a_{11}=a_{22}=\frac{-g^{\prime }(u)\left( \left( d^{2}f^{\prime
}g-c^{2}fg^{\prime }\right) \left( (f^{\prime })^{2}+(g^{\prime
})^{2}\right) +\left( g^{\prime }f^{\prime \prime }-f^{\prime }g^{\prime
\prime }\right) \left( c^{2}f^{2}+d^{2}g^{2}\right) \right) }{f(u)\left(
(f^{\prime })^{2}+(g^{\prime })^{2}\right) ^{2}\left(
c^{2}f^{2}+d^{2}g^{2}\right) }, \\ 
a_{33}=a_{44}=\frac{f^{\prime }(u)\left( \left( d^{2}f^{\prime
}g-c^{2}fg^{\prime }\right) \left( (f^{\prime })^{2}+(g^{\prime
})^{2}\right) +\left( g^{\prime }f^{\prime \prime }-f^{\prime }g^{\prime
\prime }\right) \left( c^{2}f^{2}+d^{2}g^{2}\right) \right) }{g(u)\left(
(f^{\prime })^{2}+(g^{\prime })^{2}\right) ^{2}\left(
c^{2}f^{2}+d^{2}g^{2}\right) },%
\end{array}
\label{D4}
\end{equation}%
are differentiable functions and $h_{11}^{1},h_{22}^{1}$ the coefficients of
the second fundamental form given in (\ref{C6}).
\end{lemma}

If the matrix $A$ is \ a zero matrix then $M$ becomes minimal. So minimal
rotational surfaces are of coordinate finite type.

We prove the following result.

\begin{proposition}
Let $M$ be a generalized rotation surface given by the parametrization (\ref%
{C2}). If $M$ is a \textit{coordinate finite type th}en 
\begin{equation*}
ff^{\text{ }\prime }=cgg\prime
\end{equation*}%
holds, where, $c$ is a real constant.
\end{proposition}

\begin{proof}
Suppose that the entries of the diagonal matrix $A$ are real constants.
Then, substituting the first equation in (\ref{D4}) into second one we get
the result.
\end{proof}

An easy consequence of Proposition $24$ is the following.

\begin{corollary}
Let $M$ be a Vranceanu rotation surface in Euclidean 4-space. If $M$ is a 
\textit{coordinate finite type, th}en 
\begin{equation*}
rr^{\text{ }\prime }\left( \cos ^{2}u-c\sin ^{2}u\right) =r^{2}\cos u\sin
u(1+c)
\end{equation*}%
holds, where, $c$ is a real constant.
\end{corollary}

We obtain the following result;

\begin{theorem}
Let $M$ be a Vranceanu rotation surface in Euclidean 4-space. Then $M$ is of
restricted type if and only if $M$ is one of the following:

$i)$ an open part of a Clifford torus$,$

$ii)$ a minimal surface given with the parametrization (\ref{C20}).

$iii)$ a surface given with the parametrization 
\begin{equation}
r(u)=\frac{\pm \lambda }{\sqrt{(1+c)\cos 2u+(1-c)}},c\neq 1
\end{equation}%
where, $\lambda $ and $c$ are real constants.
\end{theorem}

In \cite{Ho} C. S. Houh investigated Vranceanu rotation surfaces of finite
type and proved the following

\begin{theorem}
\cite{Ho} A flat Vranceanu rotation surface in $\mathbb{E}^{4}$ is of finite
type if and only if it is the product of two circles with the same radius,
i.e. it is a Clifford torus.
\end{theorem}

\bigskip

Beng\"{u} (K\i l\i \c{c}) Bayram

Department of Mathematics

Bal\i kesir University

Bal\i kesir, Turkey

e-mail: benguk@balikesir.edu.tr

\bigskip

Kadri Arslan \& Bet\"{u}l Bulca

Department of Mathematics

Uluda\u{g} University

16059, Bursa, Turkey

e-mail: arslan@uludag.edu.tr; bbulca@uludag.edu.tr

\bigskip

Nergiz \"{O}nen

Department of Mathematics

\c{C}ukurova University

Adana, Turkey

e-mail: nonen@cu.edu.tr

\end{document}